\documentclass[11pt]{amsart}
\usepackage{amsmath, amssymb}
\usepackage{amsfonts}
\usepackage{graphicx}
\usepackage{mathrsfs}
\usepackage[arrow,matrix,curve,cmtip,ps]{xy}
\usepackage{amsthm}
\usepackage{enumerate}
\usepackage{color}
\usepackage[colorlinks]{hyperref}

\allowdisplaybreaks

\newtheorem{thm}{Theorem}[section]
\newtheorem{prop}[thm]{Proposition}
\newtheorem{cor}[thm]{Corollary}
\newtheorem{lemma}[thm]{Lemma}

\newtheorem*{theorem*}{Theorem}
\theoremstyle{remark}

\newtheorem{defn}[thm]{Definition}

\numberwithin{equation}{section}

\newcommand{\N}{\mathbb{N}}
\newcommand{\R}{\mathbb{R}}
\newcommand{\C}{\mathbb{C}}

\begin{document}
\title[H. Larki and A. Riazi]{Stable rank of Leavitt path algebras of arbitrary graphs}

\author{Hossein Larki and Abdolhamid Riazi}

\address{Department of Pure Mathematics\\ Faculty of Mathematics
and Computer Science\\ Amirkabir University of Technology, No. 424, Hafez
Ave.,15914\\ Tehran, Iran}
\email{h.larki@aut.ac.ir (H.Larki), riazi@aut.ac.ir (A.Riazi)}


\date{\today}



\begin{abstract}
The stable rank of Leavitt path algebras of row-finite graphs was computed by Ara and Pardo. In this paper we extend this for an arbitrary directed graph. In some parts, we proceed our computation as the row-finite case while in some parts we use the knowledge about row-finite setting by applying the desingularizing method due to Drinen and Tomforde. In particular, we characterize purely infinite simple quotients of a Leavitt path algebra.
\end{abstract}

\maketitle

\section{Introduction}

For a given filed $K$, Leavitt in \cite{lea} constructed a unital $K$-algebra $L(1,n)$ generated by $\{x_1,\ldots,x_n,y_1,\ldots,y_n\}$ satisfying the relations $\sum_{i=1}^ny_ix_i=1$ and $x_iy_i=\delta_{ij}1$ for $1\leq i,j\leq n$. In \cite{abr1,ara5,abr2} the authors generalized these algebras to a certain class of path algebras $L_K(E)$ which are called \emph{Leavitt path algebras}. Leavitt path algebras $L_K(E)$ are intimately related to graph $C^*$-algebras $C^*(E)$ (see \cite{rae1}), and when $K=\C$, it is a dense $*$-subalgebra of $C^*(E)$. Leavitt path algebras also give concrete examples of many interesting classes of algebras rather than the classical Leavitt algebras $L(1,n)$, such as: (1) matrix rings $M_n(K)$ for $n\in \N\cup\{\infty\}$, where $M_\infty(K)$ denotes $\N\times\N$ matrices with only finitely many nonzero entries, (2) the Toeplitz algebra, and (3) the Laurent polynomial ring $K[x,x^{-1}]$. One of the important properties of Leavitt path algebras is structural properties of them can be related to simple properties of their underlying graphs, such as the simplicity \cite{abr1,tom2}, purely infinite simple \cite{abr3}, finite and locally finite dimensionality \cite{abr5,abr6}, semisimplicity \cite{abr4}, and the exchange property \cite{ara3}.

If $E$ is a row-finite graph which satisfies Condition (K), in \cite{ara3} it was shown that possible values for the (Bass) stable rank of $L_K(E)$ are only 1, 2, or $\infty$, and concrete criteria in terms of properties of $E$ were given for each value. In \cite{ara1}, this result was shown for all row-finite graphs. The (topological) stable rank of the graph $C^*$-algebra $C^*(E)$ was computed in \cite{dei}. Note that, in spite of many similarities between the graph $C^*$-algebra and the Leavitt path algebra of a graph, there may be a graph $E$ such that the stable rank of $C^*(E)$ is 1 while the stable rank of $L_K(E)$ is 2 (see \cite[Remark 3.3(3)]{ara1}).

In this paper we compute the stable rank of Leavitt path algebras of arbitrary graphs. Since the stable rank is not Morita invariant, we could not use the desingularizing method of Drinen and Tomforde \cite{dri} for our generalization. The computation of stable rank for arbitrary graphs largely proceeds as it does for row-finite graphs. However, we need to extend some results about the Leavitt path algebras of row-finite graphs to arbitrary graphs; in particular, we characterize purely infinite simple Leavitt path algebras.

This article is organized as follows. In Section 2, we recall some preliminaries about Leavitt path algebras which we need in the next sections. In Section 3, purely infinite simple Leavitt path algebras are characterized by giving a criterion for underlying graphs. This result is the generalization of \cite[Theorem 11]{abr3} which will be obtained by applying desingularizing method. Then we determine purely infinite simple quotients of Leavitt path algebras. Finally, in Section 4, we compute the stable rank of Leavitt path algebras.

\section{The Leavitt path algebra of a graph}

In this section we provide the basic definitions and properties of Leavitt path algebras which will be used in the next sections.

\begin{defn} A (directed) graph $E=(E^0,E^1,r,s)$ consists of a countable set of vertices $E^0$, a countable set of edges $E^1$, a source function $s:E^1\rightarrow E^0$, and a range function $r:E^1\rightarrow E^0$. A vertex $v\in E^0$ is called a \emph{sink} if $s^{-1}(v)=\emptyset$, is called \emph{finite emitter} if $0<|s^{-1}(v)|<\infty$, and is called \emph{infinite emitter} if $|s^{-1}(v)|=\infty$. If $s^{-1}(v)$ is a finite set for every $v\in E^0$, then $E$ is called \emph{row-finite}.
\end{defn}
If $e_1,\ldots,e_n$ are edges such that $r(e_i)=s(e_{i+1})$ for $1\leq i< n$, then $\alpha=e_1\ldots e_n$ is called a \emph{path} of length $|\alpha|=n$ with source $s(\alpha)=s(e_1)$ and range $r(\alpha)=r(e_n)$. For $n\geq 2$, we define $E^n$ to be the set of paths of length $n$, and $E^*:=\bigcup_{i=0}^\infty E^n$ the set of all finite paths. Note that we consider the vertices in $E^0$ to be paths of length zero. For $v,w\in E^0$, we denote $v\geq w$ if there is a path from $v$ to $w$. Also, if $X\subseteq E^0$, $v\geq X$ means there is a path from $v$ into $X$.

A \emph{closed path based at $v$} is a path $\alpha\in E^*\setminus E^0$ such that $v=s(\alpha)=r(\alpha)$. If $s(\alpha)=r(\alpha)$ and $s(e_i)\neq s(e_j)$ for $i\neq j$, then $\alpha$ is called a \emph{simple closed path}. We say that a closed path $\alpha=e_1\ldots e_n$ has an \emph{exit} if there is a vertex $v=s(e_i)$ and an edge $f\in s^{-1}(v)\setminus \{e_i\}$.

\begin{defn}\label{defn1.3}
Let $(E^1)^*$ denote the set of formal symbols $\{e^*:e\in E^1\}$. We define $v^*=v$ for all $v\in E^0$, and for a path $\alpha=e_1\ldots e_n\in E^n$ we define $\alpha^*:=e_n^*\ldots e_1^*$. We call the elements of $E^1$ \emph{real edges} and the elements of $(E^1)^*$ \emph{ghost edges}.
\end{defn}

\begin{defn}\label{defn1.4}
Let $E$ be a graph and let $R$ be a ring. A Leavitt $E$-family is a set $\{v,e,e^*: v\in E^0,e\in E^1\}\subseteq  R$ such that $\{v:v\in E^0\}$ consists of pairwise orthogonal idempotents and the following conditions are satisfied:
\begin{enumerate}[\hspace{5mm}$(1)$]
\item $s(e)e=er(e)=e$ for all $e\in E^1$,
\item $r(e)e^*=e^*s(e)=e^*$ for all $e\in E^1$,
\item $e^*f=\delta_{e,f}r(e)$ for all $e,f\in E^1$, and
\item $v=\sum_{s(e)=v}ee^*$ whenever $0<|s^{-1}(v)|<\infty$.
\end{enumerate}
\end{defn}

\begin{defn}
Let $E$ be a graph and let $K$ be a field. The \emph{Leavitt path algebra of $E$ with coefficients in $K$}, denoted by $L_K(E)$, is the universal $K$-algebra generated by a Leavitt $E$-family.
\end{defn}

The universal property of $L_K(E)$ means that if $A$ is a $K$-algebra and $\{a_v,b_e,b_{e^*}:v\in E^0, e\in E^1\}$ is a Leavitt $E$-family in $A$, then there exists a $K$-algebra homomorphism $\phi:L_K(E)\rightarrow A$ such that $\phi(v)=a_v,~\phi(e)=b_e,$ and $\phi(e^*)=b_{e^*}$ for all $v\in E^0$ and $e\in E^1$.

\begin{defn}
A graph $E$ is said to satisfy Condition (L) if every simple closed path in it has an exit, and is said to satisfy Condition (K) if any vertex in $E^0$ is either the base of no closed paths or the base of at least two distinct closed paths.
\end{defn}

Recall that a \emph{set of local units} for a ring $R$ is a set $U\subseteq  R$ of commutating idempotents with the property that for any $x\in R$ there exists $u\in U$ such that $ux=xu=x$. If $E^0$ is finite, then $1=\sum_{v\in E^0}v$ is a unit for $L_K(E)$. If $E^0$ is infinite, then $L_K(E)$ does not have a unit, but if we list the vertices of $E$ as $E^0=\{v_1,v_2,\ldots\}$ and set $t_n:=\sum_{i=1}^nv_i$, then $\{t_n\}_{n\in\mathbb{N}}$ is a set of local units for $L_K(E)$.

\begin{defn}\label{defn2.6}
A subset $X\subseteq E^0$ is called \emph{hereditary} if for any edge $e\in E^1$ with $s(e)\in X$ we have that $r(e)\in X$. Also, we say that $X\subseteq E^0$ is \emph{saturated} if for each finite emitter $v\in E^0$ with $r(s^{-1}(v))\subseteq X$ we have $v\in X$. The hereditary and saturated closure of a subset $X\subseteq E^0$, denoted by $\overline{X}$, is the smallest hereditary and saturated subset of $E^0$ containing $X$. The set of hereditary and saturated subsets of $E^0$ is denoted by $\mathcal{H}_E$.
\end{defn}

As pointed in \cite[Remark 3.1]{bat1}, for a subset $X\subseteq E^0$, we have $\overline{X}=\bigcup_{n=0}^\infty \Lambda_n(X)$, where
\begin{enumerate}[\hspace{5mm}$(1)$]
\item $\Lambda_0(X):=X\cup\{v\in E^0: \mathrm{~there~is~a~path~from~a~vertex~in}~X~\mathrm{to}~v\}$,
\item $\Lambda_{n}(X):=\Lambda_{n-1}(X)\cup\{v\in E^0:0<|s^{-1}(v)|<\infty,~r(s^{-1}(v))\subseteq\Lambda_{n-1}(X)\}$, for $n\geq 1$.
\end{enumerate}
In particular, if $H$ is a hereditary subset of $E^0$, then every vertex in $\overline{H}\setminus H$ is a finite emitter.

Suppose that $H$ is a saturated hereditary subset of $E^0$. The set of vertices in $E^0$ which emit infinitely many edges into $H$ and finitely many into $E^0\setminus H$ is denoted by $B_H$; that is
$$B_H:=\left\{v\in E^0\setminus H:|s^{-1}(v)|=\infty~\mathrm{and}~0<|s^{-1}(v)\cap r^{-1}(E^0\setminus H)|<\infty\right\}.$$
Also, for any $v\in B_H$ we denote
\[v^H:=v-\sum_{\substack{s(e)=v \\
r(e)\notin H}}ee^*.\]

\begin{defn}
If $H$ is a saturated hereditary subset of $E^0$ and $B\subseteq B_H$, then $(H,B)$ is called an \emph{admissible pair} in $E$.
If $(H,B)$ is an admissible pair in $E$, we denote $I(H,B)$ the ideal of $L_K(E)$ generated by $\{v:v\in H\}\cup\{v^H:v\in B\}$. We usually denote $I(H)$ instead of $I(H,\emptyset)$.
\end{defn}

If $(H,B)$ is an admissible pair in $E$, it is shown in \cite[Theorem 5.7]{tom2} that $L_K(E)/I(H,B)$ is isomorphic to $L_K(E/(H,B))$, where $E/(H,B)$ is the graph $E\setminus H:=(E^0\setminus H, r^{-1}(E^0\setminus H),r,s)$ with some additional edges and vertices.

The \emph{desingularization} of a graph was introduced in \cite[\S2]{tom2}, and it was shown in \cite[Theorem 5.2]{abr2} and \cite[Lemma 6.7]{tom2} that forming the desingularization preserves the Morita equivalence class of the associated Leavitt path algebra. If $E$ is a graph, the desingularization of $E$ is a row-finite graph $F$ with no sinks. If we set $p:=\sum_{v\in E^0}v$ as an idempotent in the multiplication algebra $\mathcal{M}(L_K(F))$ (see \cite[\S3]{ara4}), then $L_K(E)$ is isomorphic to the corner $p L_K(F)p$. Also, the map $I\mapsto pIp$ is a lattice isomorphism from the lattice of ideals of $L_K(F)$ onto the lattice of ideals of $L_K(E)$ (cf. \cite[Lemma 6.7]{tom2}).

In this article, we need to approach $L_K(E)$ by Leavitt path algebras of finite graphs. For this, we use the following terminology from \cite{dei}. Suppose that $E$ is an arbitrary graph and $G\subseteq E^0\cup E^1$ is a finite set. We denote $G^0:=G\cap E^0$ and $G^1=G\cap E^1$ and assume that $r(G^1)\subseteq G^0$. We define the graph $E_G$ as the following:
\begin{align*}
E_G^0&:=G^1\cup \left\{v\in G^0: ~\mathrm{either}~s^{-1}(v)=\emptyset ~\mathrm{or}~ v\in s(E^1\setminus G^1)\right\}\hspace{1.6cm}\\
E_G^1&:=\{(e,f)\in G^1\times E_G^0: r(e)=s(f)\}
\end{align*}
with $s(e,f)=e$ and $r(e,f)=f$. Note that $E_G$ is a finite graph (i.e. both the vertex set and the edge set are finite) and each vertex $v\in E_G^0\setminus G^1$ is a sink in $E_G$.

The following Lemma is the Leavitt path algebra version of \cite[Lemma 1.1]{dei}.

\begin{lemma}\label{lem2.8}
Let $E$ be a graph and let $G\subseteq E^0\cup E^1$ be a finite set with $r(G^1)\subseteq G^0$. If we set
$$\left\{
  \begin{array}{ll}
    p_v:=v-\sum_{\substack{f\in G^1 \\
    s(f)=v}}ff^* & \mathrm{if}~v\in E_G^0\setminus G^1 \\
    p_e:=ee^* & \mathrm{if}~e\in G^1 \\
    s_{(e,v)}:=e\left(v-\sum_{\substack{f\in G^1 \\
    s(f)=v}}ff^*\right) & \mathrm{if}~e\in G^1~\mathrm{and}~r(e)=v \\
    s_{(e,f)}:=eff^* & \mathrm{if}~e,f\in G^1~\mathrm{and}~r(e)=s(f),
  \end{array}
\right.$$
then $X:=\{p_v,s_e,s_e^*: v\in E_G^0,e\in E_G^1\}$ is a Leavitt $E_G$-family in $L_K(E)$ and $L_K(E_G)$ is isomorphic to the $K$-subalgebra of $L_K(E)$ generated by $X$.
\end{lemma}

\begin{proof}
It is easy to verify the set $X$ is a Leavitt $E_G$-family in $L_K(E)$. So, the universal property implies that there exists a $K$-algebra homomorphism $\phi:L_K(E_G)\rightarrow L_K(E)$ with $\phi(v)=p_v$, $\phi(e)=s_e$, and $\phi(e^*)=s_e^*$ for all $v\in E_G^0$, $e\in E_G^1$. Since $\phi(v)\neq 0$ for all $v\in E_G^0$, the Graded Uniqueness Theorem implies that $\phi$ is injective. Therefore, the range of $\phi$ is isomorphic to $L_K(E_G)$ which is a $K$-subalgebra of $L_K(E)$.
\end{proof}

Note that if $G\subseteq E^0\cup E^1$ is a finite set and we consider $L_K(E_G)$ as a subalgebra of $L_K(E)$ by Lemma \ref{lem2.8}, then $v\in L_K(E_G)$ for every $v\in G^0$.

\begin{cor}\label{cor2.9}
Let $E$ be a graph. There exists an increasing sequence $(E_i)$ of finite graphs such that $L_K(E)=\underrightarrow{\lim}L_K(E_i)$.
\end{cor}

\begin{proof}
Select an increasing family $\{G_i\}_{i\geq 1}$ of subsets of $E^0\cup E^1$ such that $\bigcup_{i=1}^\infty G_i=E^0\cup E^1$. If we denote $E_i:=E_{G_i}$, then Lemma \ref{lem2.8} implies that each $L_K(E_i)$ is a $K$-subalgebra of $L_K(E)$. Now the fact $\bigcup_i G_i=E^0\cup E^1$ concludes that $L_K(E)=\bigcup_i L_K(E_i)$.
\end{proof}

Note that $E_G$ preserves the property of having isolated closed paths (\cite[Lemma 1.3]{dei}). For example, if $E$ is acyclic, then so is $E_G$.


\section{Purely infinite simple Leavitt path algebras}

In this section, analogous to \cite[Theorem 11]{abr3} and \cite[Proposition 3.2]{bat2} we give algebraic conditions for a graph $E$ so that the associated Leavitt path algebra $L_K(E)$ is purely infinite simple. This result will be used in Section 4 to compute the stable rank of $L_K(E)$.

Recall that an idempotent $a$ in a ring $R$ is called \emph{infinite} if $aR$ is isomorphic as a right $R$-module to a proper direct summand of itself. $R$ is called \emph{purely infinite} if every nonzero right ideal of $R$ contains an infinite idempotent. By \cite[Proposition 10]{abr3}, if $R$ is a purely infinite simple ring with locall units, then for every nonzero idempotent $a\in R$ the corner $aRa$ is also purely infinite simple.

\begin{thm}\label{thm3.1}
Let $E$ be a graph. Then $L_K(E)$ is purely infinite simple if and only if

\begin{enumerate}[\hspace{5mm}$(1)$]
\item $\mathcal{H}_E=\{\emptyset,E^0\}$,
\item $E$ contains at least one closed path, and
\item $E$ satisfies Condition (L).
\end{enumerate}
\end{thm}

\begin{proof}
Suppose that $L_K(E)$ is purely infinite simple. We then have $\mathcal{H}_E=\{\emptyset,E^0\}$ and $E$ satisfies Condition (L) by \cite[Theorem 6.18]{tom2}. So, it remains to show the statement (2). If $E$ contains no closed paths, then by Corollary \ref{cor2.9} there is an increasing sequence $(E_i)$ of finite acyclic graphs such that $L_K(E)=\underrightarrow{\lim} L_K(E_i)$. Thus, \cite[Lemma 8]{abr3} implies that $L_K(E)$ is not purely infinite, a contradiction.

Conversely, assume that for graph $E$ statements (1), (2), and (3) hold. Then $L_K(E)$ is simple by \cite[Theorem 6.18]{tom2}. Take a vertex $v\in E^0$ and let $F$ be the desingularization of $E$. Recall that $L_K(E)$ and $L_K(F)$ are Morita equivalent (\cite[Lemma 6.7]{tom2}) and we have $L_K(E)=pL_K(F)p$, where $p:=\sum_{w\in E^0}w\in \mathcal{M}(L_K(F))$. Since the simplicity is Morita invariant, we see that $F$ satisfies conditions (1) and (3) by \cite[Theorem 6.18]{tom2}, while \cite[Lemm 2.6]{dri} yields condition (2) for $F$. Note that every vertex in $F$ connects to a closed path. Indeed, if $\alpha$ is a closed path in $F$ and $v\ngeq \alpha^0$, then the saturated hereditary set $\overline{\{v\}}$ does not intersect $\alpha^0$, which contradicts $\mathcal{H}_F=\{\emptyset,F^0\}$. Hence, since $F$ is row-finite, $L_K(F)$ is purely infinite simple by \cite[Theorem 11]{abr3}, and so does $vL_K(F)v$. Therefore, by the fact $vL_K(E)v=vpL_K(F)pv=vL_K(F)v$, \cite[Proposition 10(iv)]{abr3} implies that $L_K(E)$ is purely infinite simple.
\end{proof}

To prove Proposition \ref{prop3.3}, we need the following lemma.

\begin{lemma}\label{lem3.2}
Let $E$ be a graph. If $L_K(E)$ has an ideal $I$ such that $I\cap E^0=\emptyset$ and the quotient $L_K(E)/I$ is purely infinite simple, then $E$ satisfies Condition (K). In particular, we have $I=0$.
\end{lemma}

\begin{proof}
First note that $E$ has no nontrivial saturated hereditary subsets. Indeed, if $\emptyset,E^0\neq H\in \mathcal{H}_E$, then $I(H,B_H)+I$ is a nontrivial ideal in $L_K(E)/I$ which contradicts the simplicity of $L_K(E)/I$. Now, we show that $E$ satisfies Condition (K). To this end, it suffices to show that every closed path in $E$ has exits (because $E$ has no nontrivial saturated hereditary subsets). Suppose on the contrary that $\alpha$ is a closed path with no exits. Since $\mathcal{H}_E=\{\emptyset, E^0\}$, $\alpha$ is the only closed path in $E$ (up to permutation); otherwise, $\overline{\alpha^0}$ is a nontrivial saturated hereditary subset of $E^0$. Also, $E$ is row-finite. Indeed, if $v$ is an infinite emitter and $e$ is an edge with $s(e)=v$, then there exists a path from $r(e)$ to $v$ by the assumption $\mathcal{H}_E=\{\emptyset,E^0\}$. It turns out that $v$ is the base of infinity many closed paths which is impossible by the above argument. Thus \cite[Lemma 1.6]{ara1} implies that $\alpha$ has exits, a contradiction.

The second statement follows from the fact that $E$ satisfies Condition (K) if and only if every ideal in $L_K(E)$ is generated by an admissible pair in $E$ (\cite[Theorem 6.16]{tom2}).
\end{proof}

\begin{prop}\label{prop3.3}
Let $E$ be a graph. A quotient $L_K(E)/I$ is purely infinite simple if and only if there exists $H\in \mathcal{H}_E$ such that $I=I(H,B_H)$, and also
\begin{enumerate}[\hspace{5mm}$(1)$]
\item the quotient graph $E\setminus H$ contains no nontrivial saturated hereditary subsets,
\item $E\setminus H$ contains at least one closed path, and
\item $E\setminus H$ satisfies Condition (L).
\end{enumerate}
\end{prop}

\begin{proof}
Suppose that $I$ is an ideal of $L_K(E)$ such that $L_K(E)/I$ is purely infinite simple. Let $H:=I\cap E^0$ and $B:=\{w\in B_H: w^H\in I\}$, and consider the quotient $L_K(E)/I(H,B)\cong L_K(E/(H,B))$. If $\widetilde{I}$ is the image of $I$ into $L_K(E/(H,B))$ under the quotient map, then we have
$$\frac{L_K(E/(H,B))}{\widetilde{I}}\cong \frac{L_K(E)/I(H,B)}{I/I(H,B)}\cong\frac{L_K(E)}{I},$$
so $L_K(E/(H,B))/\widetilde{I}$ is purely infinite simple. Since $\widetilde{I}\cap (E/(H,B))^0=\emptyset$, Lemma \ref{lem3.2} implies that $\widetilde{I}=0$ and $I=I(H,B)$. Note that the simplicity of $L_K(E)/I\cong L_K(E/(H,B))$ yields that $B=B_H$. Now conditions (1), (2), and (3) follow from Theorem \ref{thm3.1}.

The converse implication follows from Theorem \ref{thm3.1} by considering the quotient $L_K(E)/I(H,B_H)\cong L_K(E\setminus H)$.
\end{proof}

The following definition is from \cite[Definition 1.4]{dei}.

\begin{defn}\label{defn3.4}
Suppose that $E$ is a graph and $H$ is a saturated hereditary subset of $E^0$. Let\\
\leftline{\hspace{.5cm}$F_E(H)=\{\alpha=e_1\ldots e_n\in E^*:s(\alpha)\in E^0\setminus H,$}
\rightline{$r(\alpha)\in H,~\mathrm{and}~r(e_i)\notin H~ \mathrm{for~every}~i<n\}$\hspace{.5cm}}
and let $\overline{F}_E(H)=\{\overline{\alpha}:\alpha\in F_E(H)\}$ be a copy of $F_E(H)$. Then, we define the graph $_HE$ as follows:
\begin{align*}
(_HE)^0&:=H\cup F_E(H),\hspace{7cm}\\
(_HE)^1&:=s^{-1}(H)\cup \overline{F}_E(H),
\end{align*}
with $s(\overline{\alpha})=\alpha$ and $r(\overline{\alpha})=r(\alpha)$ for $\alpha\in F_E(H)$, and the source and the range as in $E$ for the other edges of $(_HE)^1$. Note that every vertex in $(_HE)^0\setminus H$ emits one edge into $H$.
\end{defn}

In \cite[Theorem 6.1]{rui}, it was shown that every graded ideal of a Leavitt path algebra $L_K(E)$ is also a Leavitt path algebra. (In fact, this result is a modified Leavitt path algebra version of \cite[Lemma 1.6]{dei}.) As an special case of \cite[Theorem 6.1]{rui}, if $H$ is a saturated hereditary subset of $E^0$, then $I(H,\emptyset)$ is isomorphic to $L_K(_HE)$.

We need the following lemma to prove Lemma \ref{lem4.5}. The idea of its proof is from the proof of \cite[Lemma 5.4]{ara3}.

\begin{lemma}\label{lem3.5}
Let $E$ be a graph, let
$$H_0:=\left\{v\in E^0:\exists e\neq f\in E^1~\mathrm{with}~ s(e)=s(f)=v~\mathrm{and}~r(e),r(f)\geq v\right\},$$
and let $H$ be the saturated hereditary closure of $H_0$.
\begin{enumerate}[\hspace{5mm}$(1)$]
\item If $I(H)$ is unital purely infinite simple, then $L_K(E)$ has a unital purely infinite simple quotient.
\item If $I(H)$ has a unital purely infinite simple quotient, then so does $L_K(E)$.
\end{enumerate}
\end{lemma}

\begin{proof}
(1). Suppose that $I(H)$ is a unital purely infinite algebra. Then $L_K(_HE)$ is a unital purely infinite simple Leavitt algebra, and so $(_HE)^0$ is finite and the graph $_HE$ satisfies the conditions (1), (2), and (3) in Theorem \ref{thm3.1}. Let $X:=\{v\in E^0:v\geq  H\}$ and set $Y:=E^0\setminus X$. Note that $H\subseteq X$ and $X$ is a finite subset of $E^0$ because $(_HE)^0$ is finite. We claim that $Y$ is hereditary and saturated. Indeed, if $v\geq w$ and $w\in X$, then $v\geq H$ and $v\in X$, so $Y$ is hereditary. If $v\in E^0$ with $r(s^{-1}(v))\subseteq Y$, then $v$ does not connect to $H$ and hence $v\in E^0\setminus X=Y$. Thus $Y$ is also saturated.

We show that $L_K(E\setminus Y)\cong L_K(E)/I(Y,B_Y)$ is a unital purely infinite simple quotient. For this, we show that $(E\setminus Y)^0$ is finite and $E\setminus Y$ satisfies the conditions in Theorem \ref{thm3.1}. First, since $E\setminus Y=(X,r^{-1}(X),r,s)$ and $X$ is finite, $L_K(E\setminus Y)$ is unital. Also, by the construction $_HE$ in Definition \ref{defn3.4}, closed paths in $_HE$ and $(H,s^{-1}(H),r,s)$ coincide. Then $(H,s^{-1}(H),r,s)$ contains at least one closed path and every closed path in $(H,s^{-1}(H),r,s)$ has exits because $_HE$ satisfies the conditions in Theorem \ref{thm3.1}. Note that since $(_HE)^0$ is finite, each vertex $v\in X\setminus H=(E^0\setminus Y)\setminus H$ emits finitely many edges into $X$ and is the base of no closed paths. This implies that $E\setminus Y$ satisfies Condition (L). It remains to show $\mathcal{H}_{E\setminus Y}=\{\emptyset, X\}$. Suppose that $Z$ is a nonempty saturated hereditary set in $E\setminus Y=(X,r^{-1}(X),r,s)$. The simplicity of $I(H)$ yields that there is no saturated hereditary sets included in $H$. Since every vertex in $X\setminus H$ connects to $H$, we then have $H\subseteq Z$. To get a contradiction, we assume that $v\in X\setminus Z$. Since any element of $X\setminus H$ is the base of no closed paths, the relation $\geq$ is antisymmetric on $X\setminus H$. Thus the set $v^\geqslant_Z:=\{w\in X\setminus Z:v\geq w\}$ has a minimal element $w$ (because $v^\geqslant_Z$ is finite). Recall that $w$ emits finite edges into $X$. So for each finite many edges $e\in(E\setminus Y)^1$ with $s(e)=w$, we have $r(e)\in Z$ and hence $w\in Z$ by the saturation property of $Z$. This contradicts $v^\geqslant_Z\cap Z=\emptyset$, and therefore $Z=X$ as desired.

(2). Let $I(H)/J$ be unital purely infinite simple. By Proposition \ref{prop3.3}, $J$ is a graded ideal and so there exists an admissible pair $(X,B)$ in $E$ such that $J=I(X,B)$. Hence, part (1) implies that $L_K(E)/I(X,B)$ has a unital purely infinite simple quotient and so does $L_K(E)$.
\end{proof}


\section{The stable rank of Leavitt path algebras}

In this section, we compute the stable rank of Leavitt path algebras of non-row-finite graphs. As for the row-finite case in \cite[Theorem 2.8]{ara1}, we show that the possible values for the stable rank of Leavitt path algebras are 1, 2, and $\infty$.

Let $S$ be a unital ring and let $R$ be an associative ring contained in $S$ as a two-sided ideal. Recall that a vector $(a_i)_{i=1}^n$ in $S$ is called \emph{$R$-unimodular} if $a_1-1,a_i\in R$ for $i>1$ and there exist $b_1-1,b_i\in R~(i>1)$ such that $\sum_{i=1}^na_ib_i=1$. We denote the \emph{stable rank} of $R$ by $\mathrm{sr}(R)$, which is the least number $m$ for which for any $R$-unimodular vector $(a_i)_{i=1}^{m+1}$ there exist $r_i\in R$ such that the vector $(a_i+r_ia_{m+1})_{i=1}^m$ is $R$-unimodular. If such number $m$ does not exist, the stable rank of $R$ is defined infinite. Also, if $R$ is a unital ring, the \emph{elementry rank of $R$}, denoted by $\mathrm{er}(R)$, is the least natural number $n$ such that for every $t\geq n+1$ the elementary group $E_t(R)$ acts transitively on the set $U_c(t,R)$ of $t$-unimodular columns with coefficients in $R$ (cf. \cite[11.3.9]{mcc}).

The following lemma determines all Leavitt path algebras with stable rank one.

\begin{lemma}\label{lem4.1}
Let $E$ be a graph. Then $\mathrm{sr}(L_K(E))=1$ if and only if $E$ is acyclic.
\end{lemma}

\begin{proof}
Suppose that $E$ is an acyclic graph. By Corollary \ref{cor2.9}, there exists a sequence $(E_i)_{i\geq 1}$ of finite acyclic graphs such that $L_K(E)=\underrightarrow{\lim} L_K(E_i)$. Since $\mathrm{sr}(L_K(E_i))=1$ for all $i$ by \cite[Lemma 7.1]{ara3}, we have
$$\mathrm{sr}(L_K(E))\leq \liminf \mathrm{sr}(L_K(E_i))= 1,$$
and so $\mathrm{sr}(L_K(E))=1$.

For the converse, suppose that $E$ is not acyclic. If $E$ satisfies Condition (K) and we set $H_0$ and $H$ as in Lemma \ref{lem3.5}, then $H_0$ is nonempty and any vertex in $H_0$ is properly infinite as an idempotent of $L_K(E)$. Since $H_0$ generates $I(H)$, it follows that $\mathrm{sr}(I(H))> 1$ and so $\mathrm{sr}(L_K(E))>1$ by \cite[Theorem 4]{vas}.

Now assume that there is a vertex $v\in E^0$ which is the base of exactly one closed path; namely $\alpha$. If $X$ is the hereditary and saturated closure of $\{r(e): s(e)\in \alpha^0,r(e)\notin\alpha^0\}$, then $\alpha$ has no exists in $E\setminus X$. Let $Y$ be the saturated hereditary closure of $\alpha^0$ in $E\setminus X$. Note that by the construction $Y$ of $\alpha^0$ (see the paragraph following Definition \ref{defn2.6}) and the fact that $\alpha^0$ is hereditary, we see that $Y$ is row-finite and so does $_Y(E\setminus X)$. If $\tilde{I}(Y)$ is the ideal of $L_K(E\setminus X)$ generated by $Y$, then \cite[Theorem 2.8]{ara1} implies that $\mathrm{sr}(\tilde{I}(Y))=\mathrm{sr}(L_K(_Y(E\setminus X)))>1$ (because it is not acyclic). (In fact, an induction argument using \cite[Proposition 3]{abr3} shows that $L_K(_Y(E\setminus X))\cong M_n(K[x,x^{-1}])$ for some $n\in\N\cup\{\infty\}$ and so $\mathrm{sr}(L_K(_Y(E\setminus X)))=2$.) Therefore, by \cite[Theorem 4]{vas}, we get
$$\mathrm{sr}(L_K(E)) \geq \mathrm{sr}\left(\frac{L_K(E)}{I(X,B_X)}\right)=\mathrm{sr}(L_K(E\setminus X))\geq \mathrm{sr}(\tilde{I}(Y))> 1.$$
This completes the proof.
\end{proof}

\begin{lemma}\label{lem4.2}
Let $E$ be a graph. If there exists $H\in\mathcal{H}_E$ such that the quotient graph $E\setminus H$ is finite and satisfies in the conditions (1), (2), and (3) in Proposition \ref{prop3.3}, then $\mathrm{sr}(L_K(E))=\infty$.
\end{lemma}

\begin{proof}
By Proposition \ref{prop3.3}, the quotient $L_K(E)/I(H,B_H)$ is unital purely infinite simple, and so $\mathrm{sr}(L_K(E)/I(H,B_H))=\infty$ (see \cite{ara2}). Since $\mathrm{sr}(L_K(E))\geq \mathrm{sr}(L_K(E)/I(H,B_H))$ by \cite[Theorem 4]{vas}, we conclude that $\mathrm{sr}(L_K(E))=\infty$.
\end{proof}

Let $E$ be a graph. Recall that two idempotents $p,q\in L_K(E)$ are called \emph{equivalent} if there exist elements $r,s\in L_K(E)$ such that $p=rs$ and $q=sr$. For idempotents $p,q\in L_K(E)$, we write $p\lesssim q$ in the case $p$ is equivalent to an idempotent $q'$ such that $q'=q'q$. Also, as in \cite{ara5}, we denote by $V(L_K(E))$ the monoid of equivalent classes of idempotents in $M_\infty(L_K(E))$. See \cite[$\S$2 and $\S$3]{ara5} for more details.

\begin{lemma}\label{lem4.3} $\mathrm{(see}$ \cite[Lemma 6.6]{ara3} $\mathrm{and}$ \cite[Lemma 3.8]{tom-stab}$)$ Let $E$ be a graph, let $H$ be a saturated hereditary subset of $E^0$, and let $\pi:L_K(E)\rightarrow L_K(E)/I(H,\emptyset)$ be the quotient map. If $e$ is an idempotent in $L_K(E)$ and $W$ is a finite subset of $E^0\setminus H$ with $\pi(e)\lesssim\sum_{w\in W}\pi(w)$, then there exists a finite set $X\subseteq H$ such that $e\lesssim\sum_{w\in W}w+\sum_{x\in X}x$.
\end{lemma}

\begin{proof}
We claim that $V(L_K(E))/V(I(H,\emptyset))\cong V(L_K(E)/I(H,\emptyset))$. Note that when $E$ is row-finite this isomorphism holds by \cite[Theorem 2.5 and Lemma 5.6]{ara5}. Now suppose that $F$ is the desingularization of $E$ and $p:=\sum_{v\in E^0}v\in \mathcal{M}(L_K(F))$. If we denote $\tilde{I}(H)$ the ideal of $L_K(F)$ generated by $\{v:v\in H\}$, then we have

\begin{align*}
V\left(\frac{L_K(E)}{I(H,\emptyset)}\right)&\cong V\left(p\frac{L_K(F)}{\tilde{I}(H)}p\right)\\
&=\{[p(v+\tilde{I}(H))p]:v\in F^0\}\\
&\cong\{[v]+V(I(H,\emptyset)):v\in E^0\}\\
&= \frac{V(L_K(E))}{V(I(H,\emptyset))}
\end{align*}
because $I(H,\emptyset)\cong p\tilde{I}(H)p$. Thus the claim is verified. Now the rest of the proof is similar to the proof of \cite[Lemma 6.6]{ara3}.
\end{proof}

\begin{defn}
A \emph{graph trace} on $E$ is a function $g:E^0\rightarrow \R^+$ with the following two properties:
\begin{enumerate}[\hspace{5mm}$(1)$]
\item $g(v)\geq \sum_{i=1}^n g(r(e_i))$ for any vertex $v\in E^0$ and any finite set $\{e_1,\ldots,e_n\}\subseteq s^{-1}(v)$.
\item $g(v)=\sum_{s(e)=v}g(r(e))$ for any vertex $v\in E^0$ with $0<|s^{-1}(v)|<\infty$.
\end{enumerate}
We define the \emph{norm} of $g$ to be the (possibly infinite) value $\|g\|:=\sum_{v\in E^0}g(v)$. We say that $g$ is \emph{bounded} if $\|g\|<\infty$.
\end{defn}

Recall that a vertex $v$ is called \emph{left infinite} if the set $\{w\in E^0:w\geq v\}$ contains infinitely many elements. Note that if $g$ is a graph trace on $E$ and $v$ is the base of at least two distinct closed paths, then we have $g(v)=0$. Also, if $g$ is bounded and $v$ is left infinite, then $g(v)=0$.

\begin{lemma}\label{lem4.4} $\mathrm{(see}$ \cite[Corollary 6.8]{ara3}$)$
Let $E$ be a graph. If every vertex of $E$ on a simple closed path is left infinite and $E$ has no bounded graph traces, then $\mathrm{sr}(L_K(E))\leq 2$.
\end{lemma}

\begin{proof}
Write $E^0=\{v_i:i\geq 1\}$ and for each $n\geq 1$ set $p_n:=\sum_{i=1}^n v_i$. Then $\{p_n\}_{n\geq 1}$ is an increasing set of local units for $L_K(E)$. Take an arbitrary $n\in \N$ and set $V:=\{v_1,\ldots,v_n\}$. By following the proof of \cite[Proposition 6.7]{ara3} and using Lemma \ref{lem4.3} instead of \cite[Lemma 6.6]{ara3}, there exists a finite subset $W\subseteq E^0$ such that $V\cap W=\emptyset$ and $p_n=\sum_{i=1}^n v_i\lesssim\sum_{w\in W}w$. If $m:=\max\{i\in \N:v_i\in W\}$, then $m>n$ and we have $\sum_{w\in W}w\lesssim\sum_{i=n+1}^m v_i=p_m-p_n$ and so, $p_n\lesssim p_m-p_n$. Since $n\in \N$ was arbitrary, \cite[Lemma 6.1]{ara3} implies that $\mathrm{sr}(L_K(E))\leq 2$.
\end{proof}

\begin{lemma}\label{lem4.5}$\mathrm{(see}$ \cite[Lemma 3.2]{dei} $\mathrm{and}$ \cite[Lemma 7.4]{ara3}$)$
Let $E$ be a graph. If $L_K(E)$ has no unital purely infinite simple quotients, then there exists $H\in \mathcal{H}_E$ such that $E/(H,\emptyset)$ is a graph with isolated closed path (or no closed path at all) and $\mathrm{sr}(I(H,\emptyset))\leq 2$.
\end{lemma}

\begin{proof}
Let
$$H_0:=\left\{v\in E^0:\exists e\neq f\in E^1~ \mathrm{so~that~}s(e)=s(f)=v~\mathrm{and}~r(e),r(f)\geq v\right\}$$
and let $H$ be the saturated hereditary closure of $H_0$. We show that $H$ is the desired hereditary and saturated set. It is clear from the definition of $H_0$ that $E/(H,\emptyset)$ is a graph with isolated closed paths. Indeed, since additional vertices of $E/(H,\emptyset)$ to $E^0\setminus H$ are sinks, if $v$ is a vertex in $E/(H,\emptyset)$ based at least two closed paths then it belongs to $H_0$, contrary to $(E^0\setminus H)\cap H^0=\emptyset$.

We show that $\mathrm{sr}(I(H,\emptyset))\leq 2$. Since $I(H,\emptyset)\cong L_K(_HE)$, by Lemma \ref{lem4.4}, it suffices to show that every vertex in $_HE$ on a simple closed path is left infinite and $_HE$ has no bounded graph traces. These may be proved similar to the proof of \cite[Lemma 3.2]{tom-stab}. If there exists a simple closed path in $_HE$ so that its vertices are left infinite, then the argument in the third paragraph of the proof of \cite[Lemma 3.2]{dei} yields that there exists $X\in \mathcal{H}_{_HE}$ such that the graph $_HE/(X,\emptyset)$ satisfies the conditions of Theorem \ref{thm3.1}. Then by Proposition \ref{prop3.3}, $L_K(_HE)$ has a unital purely infinite simple quotient and so does $L_K(E)$ by Lemma \ref{lem3.5}, contradicting the hypothesis.

Now we show that there is no bounded graph traces on $_HE$. Suppose on the contrary that $g$ is a bounded graph trace on $_HE$. Since every vertex of $H_0$ lies on at least two closed paths, we have $H_0\subseteq N_g:=\{v\in (_HE)^0:g(v)=0\}$ and so, $H\subseteq N_g$ because $N_g$ is hereditary and saturated in $_HE$ (\cite[Lemma 3.7]{tom-stab}). It follows that $g$ is zero on all $_HE^0$ by the construction of $_HE$, a contradiction. Thus, $_HE$ satisfies the conditions of Lemma \ref{lem4.4} and so $\mathrm{sr}(L_K(_HE))\leq 2$.
\end{proof}

Now we prove the main result of this article.

\begin{thm}
Let $E$ be a graph. The values of the stable rank of $L_K(E)$ are:
\begin{enumerate}[\hspace{5mm}$(1)$]
\item $\mathrm{sr}(L_K(E))=1$ if $E$ is acyclic;
\item $\mathrm{sr}(L_K(E)=\infty$ if there exists $H\in\mathcal{H}_E$ such that the graph $E\setminus H$ is finite and satisfies the conditions of Theorem \ref{thm3.1};
\item $\mathrm{sr}(L_K(E))=2$ otherwise.
\end{enumerate}
\end{thm}

\begin{proof}
Statements (1) and (2) are Lemmas \ref{lem4.1} and \ref{lem4.2}, respectively. We thus assume that $E$ contains closed paths and does not have any $H\in \mathcal{H}_E$ such that $E\setminus H$ is finite and satisfying the conditions of Theorem \ref{thm3.1}. By Lemma \ref{lem4.5}, there exists $H\in \mathcal{H}_E$ such that $\mathrm{sr}(I(H,\emptyset))\leq 2$ and $E/(H,\emptyset)$ is a graph with isolated closed paths. Corollary \ref{cor2.9} implies that there is an increasing sequence $(E_i)$ of finite graphs with isolated closed paths such that $L_K(E/(H,\emptyset))=\underrightarrow{\lim}L_K(E_i)$. By Lemma \ref{lem2.8}, we may assume that each $L_K(E_i)$ is a $K$-subalgebra of $L_K(E/(H,\emptyset))$. Then \cite[Proposition 2.7]{ara1} yields that $\mathrm{sr}(L_K(E_i))\leq 2$ and $\mathrm{er}(L_K(E_i))=1$. Let $\pi: L_K(E)\rightarrow L_K(E)/I(H,\emptyset)\cong L_K(E/(H,\emptyset))$ be the quotient map. So, for each $i\geq 1$, we have $\pi^{-1}(L_K(E_i))/I(H,\emptyset)\cong L_K(E_i)$. If $\mathrm{sr}(L_K(E_i))=1$, then by \cite[Theorem 4]{vas}, we get
$$\mathrm{sr}(\pi^{-1}(L_K(E_i)))\leq \max\{\mathrm{sr}(I(H,\emptyset)),\mathrm{sr}(L_K(E_i))+1\}=2.$$
If $\mathrm{sr}(L_K(E_i))=2$, \cite[Corollary 2.6]{ara1} implies that $\mathrm{sr}(\pi^{-1}(L_K(E_i)))=2$. Hence $\mathrm{sr}(\pi^{-1}(L_K(E_i)))\leq 2$ for all $i\geq 1$, and since $L_K(E)=\bigcup_i \pi^{-1}(L_K(E_i))$ we conclude that
$$\mathrm{sr}(L_K(E))\leq \liminf \mathrm{sr}(\pi^{-1}(L_K(E_i)))\leq 2.$$
Now Lemma \ref{lem4.1} completes the proof.
\end{proof}


\end{document}